\documentclass[english, a4paper, 12pt]{article}
\usepackage[utf8]{inputenc}
\usepackage{multicol,amsmath,t1enc,a4,babel}
\usepackage{amsthm}
\usepackage{amssymb}
\usepackage[dvips]{graphicx}
\usepackage[dvips]{epsfig}
\usepackage{times}
\usepackage{amsfonts}
\usepackage{wasysym}
\usepackage{mathtools}
\usepackage[colorinlistoftodos]{todonotes}
\usepackage{mathrsfs}
\usepackage{relsize}
\usepackage{bbm}
\usepackage{graphicx}
\usepackage{subcaption}
\usepackage{float}
\usepackage{enumitem}

\newcommand{\e}{\mathbb{E}}

\newcommand{\p}{\mathbb{P}}
\newcommand{\zn}{\mathbb{Z}^N}
\newcommand{\po}{(S_+^n)^N}
\newcommand{\poc}{(S_+^n)^N_c}
\newcommand{\xx}{X=(X_t)_{t\in\mathbb{Z}^N}}
\newcommand{\yli}{\overset{\text{law}}{=}}

\newtheorem{defi}{Definition}[section]
\newtheorem{lemma}[defi]{Lemma}

\newtheorem{theorem}[defi]{Theorem}
\newtheorem{kor}[defi]{Corollary}

\newtheorem{rem}[defi]{Remark}
\newtheorem{ex}[defi]{Example}

\begin{document}
\title{One-to-one correspondences between discrete multivariate stationary, self-similar and stationary increment fields}

\renewcommand{\thefootnote}{\fnsymbol{footnote}}

\author{Marko Voutilainen\footnotemark[1]}

\footnotetext[1]{University of Turku, Department of Accounting and Finance, FI-20014 University of Turku, Finland. mtvout@utu.fi}

\maketitle

\begin{abstract}
In this article, we consider three important classes of $n$-variate fields indexed by the set of $N$ dimensional integers, namely stationary, stationary increment and self-similar fields. These classes are connected through bijective transformations. In addition, we introduce generalized AR$(1)$ type equations, whose unique stationary solutions are obtained via these transformations. Lastly, we apply the transformations in order to construct stationary fractional Ornstein-Uhlenbeck fields of the first and second kind.
\end{abstract}
{\small
\medskip

\medskip

}

\section{Introduction}

Stationary, self-similar and stationary increment processes form undoubtedly some of the most central classes of stochastic processes. Our focus is on processes that are indexed by the set of $N$ dimensional integers and taking values in the $n$ dimensional real space. Henceforth, we call this type of objects multivariate fields, whereas by (univariate) stochastic process we refer to collections of random variables indexed by the one dimensional set $T \in \{\mathbb{R}, \mathbb{Z}\}$, although the wider meaning of the term allows vast variation of state and parameter spaces. The cases $T = \mathbb{R}$ and $T = \mathbb{Z}$ (or more generally $T = \mathbb{R}^N$ and $T = \mathbb{Z}^N$) are called continuous and discrete, respectively. For some of the numerous potential real-world applications of the random fields that are of interest in this article, see e.g. \cite{makogin2019gaussian}, \cite{bierme2007operator} and \cite{terdik2005notes}, and references therein.

The considered type of stationarity is the strict one, meaning that the multidimensional distributions are invariant under uniform translations in the parameter space. In order to discuss stationarity of increments in the context of fields, we first need a notion for multidimensional increments. We adapt the definition applied e.g. in \cite{makogin2019gaussian}, where also different forms of stationarity of increments was investigated. Moreover, self-similarity refers to invariance of multidimensional distributions under appropriate scalings in the parameter and state spaces. The conventional notion of self-similarity of univariate processes have been extended to more generel settings in various ways. In the case of continuous time multivariate processes, different types of operator self-similarity have been considered e.g. in \cite{laha1981operator}, \cite{hudson1982operator} and \cite{sato1991self}. In \cite{genton2007self} the authors formalized the concept of multi-self-similarity for anisotropic univariate fields with the parameter space $\mathbb{R}^N$. Another approach towards anisotropy called operator-scaling random fields have been investigated e.g. in \cite{bierme2007operator}, \cite{clausel2010gaussian} and \cite{li2011multivariate}. See also \cite{kolodynski2003group}, where the notion of $G$-self-similarity allows a large spectrum of transformations of the parameter and state spaces. For an overview of univariate self-similar processes, we refer to \cite{embrechts2002selfsimilar}.

It was already shown by Lamperti in \cite{lamperti1962semi} that there exists a one-to-one correspondence between stationary and self-similar processes, which nowadays is known as Lamperti transformation. For an overview of Lamperti transformation, its applications and some of its variations, see \cite{flandrin2003stationarity}. In the case of continuous multi-self-similar univariate fields, a generalization of Lamperti transformation was introduced in \cite{genton2007self}, and in the case of discrete fields in \cite{voutilainen2023lamperti}. The two-parameter continuous case was discussed slightly earlier in \cite{terdik2005notes}. See also the recent paper \cite{davydov2019lamperti}, where Lamperti transformation was studied in a general setting related to the one in \cite{kolodynski2003group}.

It is acknowledged that Langevin (CAR$(1)$) equation can be regarded as the continuous time analogue of discrete time AR$(1)$ equation. In their classical forms, the noise can be derived from Brownian motion. One option to generalize these equations is to replace the Brownian driver with some other stationary increment noise. This approach is adapted e.g. in \cite{voutilainen2020modeling} and \cite{voutilainen2021vector}, where generalized AR$(1)$ and Langevin equations are studied for multivariate processes. The applied tools include a generalization of Lamperti transformation mapping from stationary to operator self-similar multivariate processes. Particularly, it was shown that essentially all stationary processes can be characterized as solutions to these equations. Similar results for univariate fields have been established in \cite{voutilainen2023lamperti}. 

We recall that fractional Brownian motion is self-similar and it has stationary increments being the unique centred Gaussian process possessing these properties. We immediately obtain two alternative ways to define a stationary process connected to fractional Brownian motion.
\begin{enumerate}[label=(\roman*)]
\item as the unique stationary solution to AR$(1)$ or Langevin equation driven by fBm.
\item via Lamperti transformation of fBm.
\end{enumerate}
The resulting stationary processes are called fractional Ornstein-Uhlenbeck processes of the first and second kind, respectively, and they have received a lot of attention in the literature. For details, see e.g. \cite{cheridito2003fractional}.

The remaining content of the article is organized as follows. Subsection \ref{sec:notation} introduces the applied notation. In Subsection \ref{sec:results}, we provide notions for multivariate stationary, stationary increment and self-similar fields, where the index $\Theta$ of self-similarity is a tuple of positive definite matrices. The given definition of self-similarity covers operator self-similarity of multivariate processes and multi-self-similarity of univariate fields as its special cases. A generalized Lamperti transformation $\mathcal{L}_\Theta$ gives a one-to-one correspondence between stationary fields and $\Theta$-self-similar fields. We construct transformations $\mathcal{M}_\Theta$ and $\mathcal{M}^{-1}_\Theta$ mapping from  $\Theta$-self-similar fields to stationary increment fields, and vice versa. The mapping is bijective restricted on a subset of stationary increment fields that includes e.g. multivariate fractional Brownian sheets. As a corollary, the composition $\mathcal{M}_\Theta\circ \mathcal{L}_\Theta$ gives a one-to-one correspondence between stationary fields and the class of stationary increment fields. 

Moreover, we introduce equations of AR$(1)$ type, where the noise is a general stationary increment field. We show that stationary fields can be characterized as solutions to these equations, and the connection between the stationary solution and the noise is given by the composition of the two previous transformations.

In Subsection \ref{sec:ou}, we illustrate how transformations $\mathcal{L}^{-1}_\Theta\circ\mathcal{M}^{-1}_\Theta$ and $\mathcal{L}_\Theta^{-1}$ can be utilized in construction of stationary fractional Ornstein-Uhlenbeck fields of the first and second kind.

The proofs are postponed to Section \ref{sec:proofs}.

\section{Results}
We begin with introducing the notation in Subsection \ref{sec:notation}. Subsection \ref{sec:results} consists of the essential definitions and the main results, which are applied in Subsection \ref{sec:ou} in order to formally define multivariate stationary fractional Ornstein-Uhlenbeck fields of the first and second kind.

\subsection{Notation}
\label{sec:notation}
The notation $\mathbb{N}$ stands for strictly positive integers. Let $X$ and $Y$ be random vectors of length $n$. Then the equality of the distributions is denoted by $X \yli Y$. Similarly, if $X  =(X_t)_{t\in\mathbb{Z}^N}$ and $Y = (Y_t)_{t\in\mathbb{Z}^N}$ are $n$-dimensional fields with equal finite dimensional distributions, we write $X \yli Y$. That is, for any finite collection $t_1,\dots,t_k\in\mathbb{Z}^N$ the joint distributions of $k$-tuples $\{X_{t_1}, \dots X_{t_k}\}$ and $\{Y_{t_1}, \dots Y_{t_k}\}$ of random vectors are equal. The $l^2$-vector norm and the corresponding induced matrix norm is denoted by $\Vert \cdot \Vert$. The set of symmetric $n\times n$-matrices is denoted by $S^n$. Moreover, the set of (symmetric) positive definite $n\times n$ matrices is denoted by $S_+^n$.

Occasionally, we also make use of the following notation. Let $u \subseteq \{1,\dots,N\}$ and $t\in\zn$. The cardinality of $u$ is denoted by $|u|$ and the complement of $u$ with respect to $\{1,\dots,N\}$ is denoted by $-u$. Moreover, the notation $t_u$ stands for picking the elements of $t$, whose indices belong to $u$ meaning that $t_u$ is a $|u|$-vector. For example, if $u = \{1,4,5\}$ and $t\in\mathbb{Z}^5$, then $t_u = (t_1, t_4,t_5)$. In particularly in Lemma \ref{lemma:increments} and Definition \ref{defi:bijection}, the notation is also applied in the context of multiple sums. If $j=(j_1,\dots,j_N)$ is a multi-index, then for example, we write
$$\sum_{j_u ={\bf 1}}^{t_u}$$
meaning that if $l\in u$, then $j_l$ runs from one to $t_l$ in the summation.\\

As a consequence of multidimensionality of the index and state spaces, we are working with parameters that are tuples of matrices.
\begin{defi}
By $ (S_+^n)^N$ we denote the set of $N$-tuples of positive definite $n\times n$-matrices. That is, $\Theta = (\Theta_1,\dots, \Theta_N) \in (S_+^n)^N$ if $\Theta_j\in S_+^n$ for all $j\in\{1,\dots,N\}$. We use similar notation also for $N$-tuples of different types of objects, e.g. $n$-vectors. Furthermore, $ (S_+^n)^N_c$ denotes the subset of  $(S_+^n)^N$ such that if $\Theta \in (S_+^n)^N_c$, then $\Theta$ has pairwisely commuting elements with respect to the standard matrix product.
\end{defi}
The sets $\po$ and $\poc$ are fundamental for us. For this reason, the most of definitions and results are stated under the assumption $\Theta\in\po$ (or $\Theta\in\poc$), although some of these could easily be formulated for larger classes of $N$-tuples of matrices.
Also the following bivariate operators acting on tuples of integers, vectors and matrices play an essential role in our work.

\begin{defi}
Let $t\in\zn$ and $\Theta\in (S^n)^N$. We define an operator $\ast: \zn\times (S^n)^N \to S^n$ by 
$$t\ast\Theta = \sum_{j=1}^N t_j\Theta_j.$$
In addition, if $X = (X_1, \dots, X_N)\in (\mathbb{R}^n)^N$ is a $N$-tuple of $n$-vectors, we define an operator $\star: (S^n)^N\times (\mathbb{R}^n)^N \to \mathbb{R}^N$ by
$$ \Theta\star X = \sum_{j=1}^N \Theta_j X_j. $$
\end{defi}

\begin{rem}
When $n=1$ the above definitions reduce to standard inner products of vectors: $t\ast\Theta = \langle t, \Theta \rangle$ and $\Theta \star X = \langle \Theta, X \rangle$.
\end{rem}

\subsection{Multivariate stationary, self-similar and stationary increments fields}
\label{sec:results}
We start by providing notions of stationarity and self-similarity in our setting. We consider multivariate fields that are stationary in the strict sense, and we refer to such fields simply as ''stationary fields'' in the sequal. The usual definition of stationary processes extends directly to multivariate fields.
\begin{defi}
A field $X = (X_t)_{t\in\mathbb{Z}^N} = (X_{t_1,\dots,t_N}^{(1)}, \dots, X_{t_1,\dots,t_N}^{(n)})^T_{t\in\mathbb{Z}^N}$ is stationary if
$$(X_{t+s})_{t\in\mathbb{Z}^N} \overset{\text{law}}{=} (X_t)_{t\in\mathbb{Z}^N}$$
for every $s\in\mathbb{Z}^N$.
\end{defi}
The notion of self-similarity is a more delicate matter. First of all, we need to ensure that our discrete parameter space is closed under appropriate scalings. This is done by choosing ''exponential clocks'' for self-similar fields. Secondly, we want to take account of fields that exhibit different types of self-similarity in different dimensions of the parameter and state spaces. Thus, the following definition combines the concepts of multi-self-similarity of univariate fields and operator self-similarity of multivariate processes.

\begin{defi}
\label{defi:selfsimilarity}
Let $\Theta\in (S_+^n)^N.$ A field $Y = (Y_{e^t})_{t\in\mathbb{Z}^N} = (Y^{(1)}_{e^{t_1}, \dots, e^{t_N}} , \dots, Y^{(n)}_{e^{t_1}, \dots, e^{t_N}})_{t\in\mathbb{Z}^N}$ is $\Theta$-self-similar if 
$$(Y_{e^{t+s}})_{t\in\mathbb{Z}^N} \overset{\text{law}}{=} (e^{ s\ast\Theta} Y_{e^t})_{t\in\mathbb{Z}^N}$$
for every  $s\in\mathbb{Z}^N$.
\end{defi}
Next, we discuss how Definition \ref{defi:selfsimilarity} is connected to definitions existing in the literature. Under Definition \ref{defi:selfsimilarity},
$$Y_{e^{t_1}, \dots, e^{t_j+s_j},\dots,e^{t_N}}\overset{\text{law}}{=} e^{s_j\Theta_j} Y_{e^{t_1},\dots, e^{t_N}}.$$ 
Let $Q_j\Lambda_jQ_j^T$ be an eigendecomposition of $\Theta_j$ with eigenvalues $\lambda_{j,k}$, then
$$e^{s_j\Theta_j} = Q_j \mathrm{diag}\left(e^{s_j \lambda_{j,k}} \right) Q_j^T \rightarrow 0$$
as $s_j\to -\infty$. Consequently, under the usual assumption of stochastic continuity, a self-similar field $Y_{r_1,\dots,r_N}$ that is defined also on the hyperplanes $r_j=0$ satisfies $Y_{r_1,\dots,r_N} = 0$ on these hyperplanes, as it is expected.

Let us consider Definition \ref{defi:selfsimilarity} with the index set $\mathbb{R}^N$. We note that if $\Theta\in \poc$, then change of variables $e^{{t_j}}=r_j$ and $e^{s_j} = a_j$ yield
\begin{equation}
\label{continuoustime}
\left(Y_{a_1r_1,\dots, a_Nr_N} \right)_{r\in(0,\infty)^N} \overset{\text{law}}{=}  a_1^{\Theta_1}\dots a_N^{\Theta_N} (Y_r)_{r\in (0,\infty)^N}.
\end{equation}
When $n=1$, this corresponds to the definition of multi-self-similarity of univariate fields in  \cite{genton2007self}, where self-similarity indices $\Theta_j$ are positive real numbers representing self-similarity in different dimensions of the parameter space.\\
For index set $\mathbb{R}$ ($N=1$), Definition \ref{defi:selfsimilarity} gives
$$(Y_{ar})_{r\in (0,\infty)} \overset{\text{law}}{=} a^\Theta (Y_r)_{r\in(0,\infty)},$$
which corresponds to a type of operator self-similarity of vector-valued processes. This extends in a natural way the conventional notion of self-similarity of scalar random processes, where $n=N=1$ and $\Theta$ is a positive real number. For notions of operator self-similarity, see \cite{sato1991self} and references therein. For an account on self-similarity in the basic setting, we refer to \cite{embrechts2002selfsimilar}.

A well-known example of a scalar-valued multi-self-similar field is the fractional Brownian sheet (fBs) $B^H = (B^H_t)_{t\in\mathbb{R}^N}$ with Hurst multi-index $H = (H_1, \dots, H_N)$, $0 < H \leq 1$, introduced in \cite{kamont1995fractional}. FBs is the centred Gaussian field with covariance
$$\e(B^H_t B^H_s) = \frac{1}{2^N} \prod_{j=1}^N \left(|t_j|^{2H_j} + |s_j|^{2H_j}- |t_j-s_j|^{2H_j} \right), $$
and it is multi-self-similar with $\Theta = H$. Note that $B^H_t=0$ (a.s) on the hyperplanes with $t_l=0$ for some $l = 1,\dots,N$.

Let $\Theta_j = \mathrm{diag}(\Theta_{j}^{(k)})$ with positive entries. Then from \eqref{continuoustime}, we get
$$\left(Y_{a_1r_1,\dots, a_Nr_N} \right)_{r\in(0,\infty)^N} \overset{\text{law}}{=} \mathrm{diag} \left( \prod_{j=1}^N a_j^{\Theta_{j}^{(k)}} \right)(Y_r)_{r\in (0,\infty)^N}.$$
That is, the $k$th component of $Y$ is multi-self-similar with $\Theta^{(k)} = (\Theta_{1}^{(k)}, \dots, \Theta_{N}^{(k)})$. When $0< \Theta_j^{(k)} \leq 1$, this is satisfied by $N$ independent fractional Brownian sheets with Hurst multi-indices $H^{(k)} = \Theta^{(k)}$. Furthermore, the discretized versions of the above-discussed fields obviously satisfy Definition \ref{defi:selfsimilarity}.\\

Similarly as the notion of self-similarity, the next introduced Lamperti transformation provides definitions acknowledged in the literature as its special cases.

\begin{defi}
Let $\Theta\in\po$, and $X=(X_t)_{t\in\zn}$ and $Y= (Y_{e^t})_{t\in\zn}$. We set
\begin{align*}
(\mathcal{L}_\Theta X)_{e^t} &= e^{ t\ast\Theta} X_t, \quad t\in\mathbb{Z}^N\\
(\mathcal{L}^{-1}_\Theta Y)_t &= e^{- t\ast\Theta} Y_{e^t}, \quad  t\in\mathbb{Z}^N.
\end{align*}
\end{defi}
The transformation $\mathcal{L}_\Theta$ gives a one-to-one connection between stationary fields and $\Theta$-self-similar fields.

\begin{theorem}
\label{theorem:lamperti}
Let $\Theta\in\poc$. If $X$ is stationary, then $\mathcal{L}_\Theta X$ is $\Theta$-self-similar. Conversely, if $Y$ is $\Theta$-self-similar, then $\mathcal{L}^{-1}_\Theta Y$ is stationary. Moreover, 
$$(\mathcal{L}^{-1}_\Theta \circ \mathcal{L}_\Theta)(X) = X\quad\text{and}\quad (\mathcal{L}_\Theta \circ \mathcal{L}^{-1}_\Theta)(Y)=Y$$
for all stationary $X$ and $\Theta$-self-similar $Y$.
\end{theorem}

Next, we turn to stationary increment fields. To this end, we first provide the notions of rectangular and unit cube increments, the latter being a special case of the former. The names of the two types of increments refer to their geometry in the index space.
\begin{defi}
Let $t,s\in\mathbb{Z}^N$ and $X=(X_t)_{t\in\zn}$. The rectangular increment of $X$ corresponding to the hyperrectangle $[s,t]$ is
\begin{equation}
\label{eq:increments}
\Delta_{s,t} X = \sum_{(i_1,\dots,i_N)\in\{0,1\}^N} (-1)^{\sum_{l=1}^N i_l} X_{t_1-i_1(t_1-s_1),\dots,t_N-i_N(t_N-s_N)}= \sum_{i\in\{0,1\}^N}(-1)^{\sum i} X_{t-i(t-s)}.
\end{equation}
Similarly, for $Y=(Y_{e^t})_{t\in\zn}$, we define
$$\Delta_{s,t} Y = \sum_{(i_1,\dots,i_N)\in\{0,1\}^N} (-1)^{\sum_{l=1}^N i_l} Y_{e^{t_1-i_1(t_1-s_1)},\dots,e^{t_N-i_N(t_N-s_N)}}= \sum_{i\in\{0,1\}^N}(-1)^{\sum i} Y_{e^{t-i(t-s)}}.$$
\end{defi}

\begin{defi}
Let $t\in\mathbb{Z}^N$ and $X=(X_t)_{t\in\zn}$. The unit cube increment of $X$ at $t$ is given by
\begin{equation}
\label{eq:increments2}
\Delta_t X = \Delta_{t-1,t} X = \sum_{(i_1,\dots,i_N)\in\{0,1\}^N} (-1)^{\sum_{l=1}^N i_l} X_{t_1-i_1,\dots,t_N-i_N}= \sum_{i\in\{0,1\}^N}(-1)^{\sum i} X_{t-i}.
\end{equation}
Similarly, for $Y=(Y_{e^t})_{t\in\zn}$,
$$\Delta_t Y = \Delta_{t-1,t} Y = \sum_{i\in\{0,1\}^N}(-1)^{\sum i} Y_{e^{t-i}}.$$
\end{defi}

\begin{rem}
\label{rem:properties}
When $N=1$, the rectangular increment $\Delta_{s,t}X$ is just the usual one-dimensional increment $X_t-X_s$. Due to the alternating sign in \eqref{eq:increments}, degenerate increments ($t_l =s_l$ for some $l$) are equal to zero. E.g. in \cite{makogin2019gaussian} the authors considered rectangular increments with $s \leq t$ element-wise. However, it can be verified from \eqref{eq:increments} that the multidimensional increments satisfy yet another property that is analogous with the one-dimensional increments: Let $\tilde{s}$ and $\tilde{t}$ be the vectors obtained from $s$ and $t$ by swapping the elements for which $s_l > t_l$, and let $m$ be the number of such elements. Thus, now $\tilde{s} \leq \tilde{t}$ and in addition,
$$\Delta_{s,t}X = (-1)^m \Delta_{\tilde{s},\tilde{t}} X.$$
\end{rem}

\begin{ex}
In the two-parameter case, we obtain for example that
$$\Delta_{s,t} X = X_{t_1,t_2}- X_{s_1,t_2}- X_{t_1,s_2} + X_{s_1,s_2}$$
and
$$\Delta_t Y = Y_{e^{t_1}, e^{t_2}} - Y_{e^{t_1},e^{t_2-1}} - Y_{e^{t_1-1},e^{t_2}}+Y_{e^{t_1-1},e^{t_2-1}}.$$
\end{ex}

The following lemma shows how rectangular increments can be constructed from unit square increments.
\begin{lemma}
\label{lemma:incrementsum}
Let $t,s\in\zn$ with $s \leq t$ element-wise. Then 
$$\sum_{j_1=s_1+1}^{t_1}\dots \sum_{j_N = s_N+1}^{t_N} \Delta_j Z  \eqqcolon \sum_{j=s+{\bf 1}}^t \Delta_j Z = \Delta_{s,t} Z,$$
where $Z$ is a field indexed by $\zn$ and sums of the type $\sum_{j_k=t_k+1}^{t_k}$ are empty sums.
\end{lemma}

By saying that $(X_t)_{t\in\zn}$ is a stationary increment field we simply mean that the unit cube increment field $(\Delta_t X)_{t\in\zn}$ is stationary.
\begin{defi}
\label{defi:stationaryincrements}
A field $X = (X_t)_{t\in\mathbb{Z}^N}$ is a stationary increment field if 
$$(\Delta_{t+s} X)_{t\in\mathbb{Z}^N} \overset{\text{law}}{=} (\Delta_t X)_{t\in\mathbb{Z}^N}$$
for every $s\in\mathbb{Z}^N$.
\end{defi}

\begin{rem}
In \cite{makogin2019gaussian} the authors consider several notions of stationarity of increments. In particularly, their notion of strictly stationary rectangular increments corresponds to invariance of joint probability distributions of finite collections of rectangular increments under uniform translations in the parameter space. In our setting, this can be formalized as
$$
\begin{pmatrix}
\Delta_{s^{(1)}+h, t^{(1)}+h} X\\
\vdots \\
\Delta_{s^{(m)}+h, t^{(m)}+h} X
\end{pmatrix} 
\overset{\text{law}}{=} 
\begin{pmatrix}
\Delta_{s^{(1)}, t^{(1)}} X\\
\vdots\\
\Delta_{s^{(m)}, t^{(m)}} X
\end{pmatrix},
$$
where $m\in\mathbb{N}$, $h\in\mathbb{Z}^N$ and $t^{(j)}, s^{(j)} \in\mathbb{Z}^N$ for all $j=1,\dots,m$.
However, since unit cubes provide basic building blocks for our discrete parameter space in the fashion of Lemma \ref{lemma:incrementsum}, the above is equivalent with Definition \ref{defi:stationaryincrements}.
\end{rem}
A (discrete) fractional Brownian sheet provides an example of a stationary increment field. In addition, unlike in the case of processes and fractional Brownian motion, fBs is not the only Gaussian self-similar field possessing stationary increments, as the example provided in \cite{makogin2015example} reveals.

Stationary increment fields satisfying a rather mild additional property form the third essential class of multivariate fields, besides of stationary fields and self-similar fields, considered in this article. We show that the existence of a certain logarithmic moment is a sufficient condition and hence, the class includes e.g. all Gaussian stationary increment fields.

\begin{defi}
\label{defi:class}
Let $\Theta\in \poc$ and let $G= (G_t)_{t\in\mathbb{Z}^N}$ be a stationary increment field. If
\begin{equation}
\label{eq:glimitcondition}
\lim_{M_1\to\infty} \dots \lim_{M_N\to\infty} \sum_{j_1=-M_1}^{t_1}\dots \sum_{j_N=-M_N}^{t_N}  e^{j\ast \Theta} \Delta_{j}G
\end{equation}
converges in probability for every $t\in\mathbb{Z}^N$, then we write $G\in\mathcal{G}_\Theta$. In this case, we denote the limit of \eqref{eq:glimitcondition} as
\begin{equation}
\label{limitnotation}
\sum_{j_1=-\infty}^{t_1}\dots \sum_{j_N=-\infty}^{t_N}  e^{j\ast \Theta} \Delta_{j}G =  \sum_{j=-\infty}^t e^{j\ast \Theta} \Delta_{j}G.
\end{equation}
Moreover, if $G_t = 0$ (a.s.) whenever $t_l=0$ for some $l\in\{1,\dots,N\}$, then we write $G\in\mathcal{G}_{\Theta,0} \subset \mathcal{G}_{\Theta}$.
\end{defi}

\begin{rem}
\label{rem:permutations}
It turns out that if \eqref{eq:glimitcondition} exists for any permutation of limits, then it exists for all permutation of limits and as a $N$-fold limit with a common limiting random vector (see also Corollary \ref{kor:permutations}). This justifies the abbreviated notation \eqref{limitnotation}. 
\end{rem}

\begin{lemma}
\label{lemma:logarithm}
Let $G = (G_t)_{t\in\mathbb{Z}^N}$ be a stationary increment field. Assume that
\begin{equation}
\label{logarithmic}
\e \left(\ln \Vert \Delta_{{\bf 1}} G \Vert \mathbbm{1}_{\{ \Vert \Delta_{{\bf 1}} G\Vert > 1 \}} \right)^{N+ \delta} < \infty
\end{equation}
for some $\delta >0$, where $\Delta_{{\bf 1}} = \Delta_{1,\dots,1}$. Then \eqref{eq:glimitcondition} converges almost surely for all $\Theta\in\poc$. In particularly, $G \in\mathcal{G}_{\Theta}$. 

Moreover, if $G_t = 0$ whenever $t_l=0$ for some $l\in\{1,\dots,N\}$, then $\Vert \Delta_{\bf 1} G \Vert$ can be replaced by $\Vert G_{{\bf 1}}\Vert = \Vert G_{1,\dots,1} \Vert$ in \eqref{logarithmic}. In this case, $G \in\mathcal{G}_{\Theta,0}$ for all $\Theta\in\poc$.
\end{lemma}

\begin{kor}
\label{kor:moment}
Let $G = (G_t)_{t\in\mathbb{Z}^N}$ be a stationary increment field. If
$$\e \Vert \Delta_{{\bf 1}}G \Vert < \infty,$$
then $G\in\mathcal{G}_\Theta$ for all $\Theta\in\poc$. In particularly, this holds if $\e \Vert G_t \Vert < \infty$ for all $t\in\zn$.
\end{kor}
The following lemma shows that if $G\in\mathcal{G}_\Theta$, then \eqref{eq:glimitcondition} defines a $\Theta$-self-similar field.

\begin{lemma}
\label{lemma:selfsimilar}
Let $\Theta\in\poc$ and $G\in\mathcal{G}_\Theta$. Define
\begin{equation}
\label{defY}
Y_{e^t} = \sum_{j=-\infty}^{t} e^{j\ast \Theta} \Delta_j G,
\end{equation}
then $\Delta_t Y = e^{t\ast \Theta} \Delta_t G$ for all $t\in\zn$. Moreover, $Y$ is $\Theta$-self-similar.
\end{lemma}

\begin{rem}
In the proof of the first claim of \ref{lemma:selfsimilar}, the assumptions $\Theta\in\poc$ and $G\in\mathcal{G}_\Theta$ are used only to ensure that $Y_{e^t}$ defined by \eqref{defY} makes sense as a random vector. If we regard \eqref{defY} as a formal series, then these assumptions can be dropped still obtaining $\Delta_tY= e^{t\ast \Theta} \Delta_t G$ for all $t\in\zn$.
\end{rem}

The combination of the next two lemmas gives us means to construct fields belonging to $\mathcal{G}_\Theta$ from $\Theta$-self-similar fields.
\begin{lemma}
\label{lemma:stationaryincrements}
Let $\Theta\in\po$ and $Y$ be $\Theta$-self-similar. If for $G=(G_t)_{t\in\zn}$ it holds that 
$$\Delta_tG = e^{-t\ast\Theta} \Delta_tY, \quad t\in\zn,$$
then 
\begin{equation*}
\sum_{j=-\infty}^t e^{j \ast \Theta} \Delta_j G = Y_{e^t}.
\end{equation*}
Moreover, if $\Theta\in\poc$, then $G$ is stationary increment field and thus $G \in\mathcal{G}_\Theta$.
\end{lemma}

\begin{kor}
\label{kor:permutations}
Assume that $G \in\mathcal{G}_\Theta$ according to Definition \ref{defi:class}. By Lemma \ref{lemma:selfsimilar} the limiting field is $\Theta$-self-similar with $\Delta_t Y = e^{t\ast \Theta} \Delta_t G$. By multiplying with the inverse matrix $e^{-t \ast \Theta}$ from the left, we obtain $\Delta_t G = e^{-t \ast \Theta} \Delta_t Y$. Now the proof of Lemma \ref{lemma:stationaryincrements} shows that 
$$ \lim_{M\to\infty} \sum_{j=-M}^t e^{j\ast \Theta} \Delta_j G = Y_{e^t}$$
regardless of how we approach the limit. This verifies the claim of Remark \ref{rem:permutations}.
\end{kor}

In \eqref{G}, we apply the notation introduced at the beginning of Subsection \ref{sec:notation}.
\begin{lemma}
\label{lemma:increments}
Let $t\in\zn$ and let $u\subseteq \{1,\dots,N\}$ be the set of indices $l$ for which $t_l \geq 0$. Moreover, let $Y$ be a field indexed by $\mathbb{Z}^N$. We define
\begin{equation}
\label{G}
G_t = (-1)^{|-u|} \sum_{j_u = {\bf 1}}^{t_u} \sum_{j_{-u} = t_{-u}+{\bf 1}}^{ \bf 0} e^{-j\ast\Theta} \Delta_j Y,
\end{equation}
where $j=(j_1,\dots,j_N)$ is a multi-index and $\Theta$ is an arbitrary $N$-tuple of $n\times n$-matrices. Now
$$\Delta_t G = e^{-t\ast \Theta} \Delta_t Y\quad\text{for all } t\in\mathbb{Z}^N$$
and
$G_t = 0$ for all $t$ such that $t_l=0$ for some $l\in\{1,\dots,N\}$.
\end{lemma}

\begin{rem}
\label{rem:modifications}
Definition \eqref{G} can straightforwardly be modified in such a way that $\Delta_t G = e^{- t\ast\Theta}\Delta_t Y$ still holds and $G_t$ vanishes on hyperplanes $t\in\zn$ with $t_l = c_l$ for some $l$, and $c_l \in\mathbb{R}$. In addition, if we set $\tilde{G} = G + \xi$, where $\xi$ is a random variable on the underlying probability space, we obtain that $\Delta_t \tilde{G} = e^{- t\ast\Theta} \Delta_t Y$ and $\tilde{G}_t = \xi$ for $t\in\zn$ such that $t_l = c_l$ for some $l$. See also \cite{voutilainen2023lamperti}, where $G_t$ defined from a self-similar $Y$ with $\Delta_t G = e^{- t\ast\Theta} \Delta_t Y$ vanishes on hyperplanes $\sum_{l=1}^N t_l \in\{0,-1,\dots,-N+1\}$. Thus, under $\Theta\in\poc$ and by Lemma \ref{lemma:stationaryincrements}, we obtain stationary increment fields belonging to different subsets of $\mathcal{G}_\Theta$.
\end{rem}

Definition \ref{defi:bijection} and Theorem \ref{theorem:bijection} encapsulate the essence of Lemmas \ref{lemma:selfsimilar}, \ref{lemma:stationaryincrements} and \ref{lemma:increments}. 
\begin{defi}
\label{defi:bijection}
Let $\Theta\in\poc$, and $G=(G_t)_{t\in\mathbb{Z}^N}$ and $Y=(Y_{e^t})_{t\in\mathbb{Z}^N}$. We define
$$(\mathcal{M}_\Theta Y)_t = (-1)^{|-u|} \sum_{j_u = {\bf 1}}^{t_u} \sum_{j_{-u} = t_{-u}+{\bf 1}}^{\bf 0} e^{-j\ast\Theta} \Delta_j Y$$
and
$$(\mathcal{M}^{-1}_\Theta G)_{e^t} = \sum_{j=-\infty}^{t} e^{j\ast \Theta} \Delta_j G.$$
\end{defi}

The transformation $\mathcal{M}_\Theta$ gives a one-to-one connection between $\Theta$-self-similar fields and fields in $\mathcal{G}_{\Theta,0}$.
\begin{theorem}
\label{theorem:bijection}
If $Y$ is $\Theta$-self-similar, then $\mathcal{M}_\Theta Y \in\mathcal{G}_{\Theta,0}$. If $G\in\mathcal{G}_\Theta$, then $\mathcal{M}^{-1}_\Theta G $ is $\Theta$-self-similar. Moreover,
$$(\mathcal{M}_\Theta^{-1} \circ\mathcal{M}_\Theta) (Y) = Y \quad\text{and}\quad (\mathcal{M}_\Theta \circ \mathcal{M}_\Theta^{-1}) (G) = G$$
for all $\Theta$-self-similar $Y$ and $G\in\mathcal{G}_{\Theta,0}$.
\end{theorem}

\begin{kor}
If $X$ is stationary, then $(\mathcal{M}_\Theta \circ \mathcal{L}_\Theta)(X)\in\mathcal{G}_{\Theta,0}$. If $G \in\mathcal{G}_\Theta$, then $(\mathcal{L}_\Theta^{-1} \circ \mathcal{M}_\Theta^{-1})(G)$ is stationary.
Moreover, $\mathcal{M}_\Theta \circ \mathcal{L}_\Theta$ is a bijection between stationary fields and fields in $\mathcal{G}_{\Theta,0}$.
\end{kor}

\begin{rem}
\label{rem:alternatives}
By carefully examining the proof of Theorem \ref{theorem:bijection}, the following observations can be made. Let $P$ be a predicate and let $\mathcal{G}_{\Theta, P}\subset \mathcal{G}_\Theta $ consist of the fields satisfying $P$. Assume that the definition of $\mathcal{M}_\Theta$ is modified in such a way that $\Delta_t \mathcal{M}_\Theta Y = e^{-t\ast \Theta} \Delta_t Y$ still holds and $\mathcal{M}_\Theta Y\in \mathcal{G}_{\Theta, P}$. Now, 
$$(\mathcal{M}_\Theta^{-1} \circ\mathcal{M}_\Theta) (Y) = Y$$ 
for all self-similar $Y$. Moreover, if $P$ together with fixed increments define a unique field in $\mathcal{G}_{\Theta, P}$, then 
$$(\mathcal{M}_\Theta \circ\mathcal{M}^{-1}_\Theta) (G) = G$$ 
for all $G\in\mathcal{G}_{\Theta,P}$. That is, in Theorem \ref{theorem:bijection} we obtain a bijection between the sets of $\Theta$-self-similar fields and $\mathcal{G}_{\Theta,P}$. See also Remark \ref{rem:modifications}, where examples of such sets $\mathcal{G}_{\Theta,P}$ are discussed.

Furthermore, the remaining theorems hold true for such modified transformations $\mathcal{M}_\Theta$ when the class $\mathcal{G}_{\Theta,0}$ is replaced with the appropriate $\mathcal{G}_{\Theta, P}$.
\end{rem}

Lastly, we examine solutions of equations that can be interpreted as generalized AR$(1)$ equations, where the noise field is of class $\mathcal{G}_\Theta$. To this end, we first define the ''AR$(1)$ part'' that consists of random vectors that are previous at least in one coordinate direction of the parameter space. 
\begin{defi}
\label{defi:hassu}
Let $\Theta\in\po$ and $\xx$. Then, for every $i\in\{0,1\}^N$ with $i\neq \bf{0}$, we define 
$$\hat{\Theta}_i =  (-1)^{1+\sum_{l=1}^N i_l}e^{- i\ast\Theta} \quad\text{and}\quad  \hat{X}_{t,i} = X_{{t_1-i_1},\dots,{t_N-i_N}}, \quad t\in\mathbb{Z}^N.$$
Furthermore, let $\hat{\Theta} = (\hat{\Theta}_i) \in\ (S^n)^{2^N-1}$ and $\hat{X}_t = (\hat{X}_{t,i})\in (\mathbb{R}^n)^{2^N-1}$, where the elements of $\hat{\Theta}$ and $\hat{X}_t$ are ordered consistently. Now
$$\hat{\Theta}\star\hat{X}_t= \sum_{\substack{(i_1,\dots,i_N)\in\{0,1\}^N\\ i\neq \bf{0}}} (-1)^{1+\sum_{l=1}^N i_l} e^{- i\ast\Theta} X_{{t_1-i_1},\dots,{t_N-i_N}}.$$
\end{defi}

\begin{ex}
The cases $N=1$ and $N=2$ yield
$$\hat{\Theta}\star\hat{X}_t = e^{-\Theta} X_{t-1}$$
and
$$\hat{\Theta}\star\hat{X}_t = e^{-\Theta_1}X_{t_1-1, t_2} +e^{-\Theta_2}X_{t_1, t_2-1} - e^{-\Theta_1-\Theta_2}X_{t_1-1, t_2-1},$$
respectively.
\end{ex}

The proofs of the next two theorems follow closely the lines of the corresponding proofs for univariate fields in \cite{voutilainen2023lamperti}. However, for the reader's convenience, they are presented at the end of Section \ref{sec:proofs}. The first theorem shows that generalized AR$(1)$ equations admit stationary solutions, whereas the second one shows that stationary fields solve this type of equations. The connection between the stationary solution and the noise field is given by the composition $ \mathcal{M}_\Theta \circ \mathcal{L}_\Theta$.
\begin{theorem}
\label{theorem:first}
Let $\Theta\in\poc$ and $\xx$. If 
\begin{enumerate}[label=(\roman*)]
\item $$\lim_{m\to -\infty} e^{m\Theta_j}X_{t_1,\dots,t_{j-1},m,t_{j+1},\dots,t_N} \overset{\p}{\longrightarrow} 0$$
for every $j\in\{1,\dots,N\}$ and $t_1, \dots,t_{j-1},t_{j+1},\dots, t_N \in \mathbb{Z}$.
\item For $G\in\mathcal{G}_\Theta$, 
$$X_t = \hat{\Theta}\star\hat{X}_t + \Delta_t G, \quad t\in\zn,$$
\end{enumerate}
then $X$ is stationary with $X = (\mathcal{L}_\Theta^{-1} \circ \mathcal{M}_\Theta^{-1})(G)$. 
\end{theorem}

\begin{rem}
$X = (\mathcal{L}_\Theta^{-1} \circ \mathcal{M}_\Theta^{-1})(G)$ is the unique stationary solution to 
$$X_t = \hat{\Theta}\star\hat{X}_t + \Delta_t G, \quad t\in\zn.$$
\end{rem}

\begin{theorem}
\label{theorem:second}
Let $\Theta \in \poc$ and let $X= (X_t)_{t\in\mathbb{Z}^N}$ be stationary. Set $G  =(\mathcal{M}_\Theta \circ \mathcal{L}_\Theta)(X)$. Then 
 \begin{equation}
 \label{ar1}
 X_t = \hat{\Theta} \star \hat{X}_t  + \Delta_t G\quad\text{for all } t\in\zn.
 \end{equation}
\end{theorem}

The last theorem characterizes stationary fields as solutions to \eqref{ar1}, where the noise field belongs to $\mathcal{G}_{\Theta,0}$.
\begin{theorem}
\label{theorem:main}
Let $\Theta\in\poc$. Then $X=(X_t)_{t\in\mathbb{Z}^N}$ is stationary if and only if

\begin{enumerate}[label=(\roman*)]
\item $$\lim_{m\to -\infty} e^{m\Theta_j}X_{t_1,\dots,t_{j-1},m,t_{j+1},\dots,t_N} \overset{\p}{\longrightarrow} 0$$
for every $j\in\{1,\dots,N\}$ and $t_1, \dots,t_{j-1},t_{j+1},\dots, t_N \in \mathbb{Z}$.
\item  $$X_t = \hat{\Theta}\star\hat{X}_t + \Delta_t G, \quad t\in\zn$$
with $G\in\mathcal{G}_{\Theta,0}$.
\end{enumerate}
Moreover, the fields are connected as $X =  (\mathcal{L}_\Theta^{-1} \circ \mathcal{M}_\Theta^{-1})(G)$ and $G = (\mathcal{M}_\Theta \circ \mathcal{L}_\Theta)(X)$.
\end{theorem}

\subsection{Multivariate stationary fractional Ornstein-Uhlenbeck fields}
\label{sec:ou}
Let $A$ be a $n\times n$-matrix and let $B_t = (B_t^{H^{(1)}}, \dots B_t^{H^{(n)}})_{t\in\zn}$ consist of independent fractional Brownian sheets with Hurst indices $H^{(k)} = (H_1^{(k)}, \dots, H_N^{(k)})$, $0<H_j^{(k)}\leq 1.$ Set $G_t = A B_t$. Then $\Delta_t G = A \Delta_t B$ and $G$ is clearly a stationary increment field, and since $G$ is Gaussian, it belongs to $\mathcal{G}_\Theta$ for every $\Theta\in\poc$ by Lemma \ref{lemma:logarithm}. This allows us to define (discrete) stationary fractional Ornstein-Uhlenbeck fields of the first kind as
$$X = (\mathcal{L}^{-1}_\Theta\circ\mathcal{M}^{-1}_\Theta)(G).$$
Set $\Theta = (\Theta_1, \dots, \Theta_N)$, where $\Theta_j = \mathrm{diag}(H_j^{(k)})$. In the light of the discussion after Definition \ref{defi:selfsimilarity},
$$(B_{e^{t+s}})_{t\in\zn} \overset{law}{=} \mathrm{diag}\left(\prod_{j=1}^N e^{s_j H_j^{(k)}} \right) (B_{e^t})_{t\in\zn} = e^{s\ast \Theta}(B_{e^t})_{t\in\zn}. $$
That is, $B$ is $\Theta$-self-similar. Furthermore, set $Y_{e^t} = AB_{e^t}$. Then, under the assumption that $A$ and $e^{s\ast \Theta}$ commute,
$$(Y_{e^{t+s}})_{t\in\zn} = A(B_{e^{t+s}})_{t\in\zn} \overset{law}{=}A e^{s\ast \Theta}(B_{e^t})_{t\in\zn} = e^{s\ast \Theta}(AB_{e^t})_{t\in\zn} =e^{s\ast \Theta} (Y_{e^t})_{t\in\zn}. $$
showing that also $Y$ is $\Theta$-self-similar. Note that this holds e.g. when $A$ is diagonal or $H^{(k)} = H =(H_1, \dots, H_N)$, since in the latter case
$$e^{s\ast \Theta} = \mathrm{diag}\left(\prod_{j=1}^N e^{s_j H_j} \right)= I\prod_{j=1}^N e^{s_j H_j}  . $$
This approach allows us to define (discrete) stationary fractional Ornstein-Uhlenbeck fields of the second kind as
$$X = \mathcal{L}^{-1}_\Theta Y.$$
Lastly, we point out that the above can be applied also e.g. to the Gaussian self-similar field with stationary increments that is not a fBs introduced in \cite{makogin2015example}. In this way, we are able to construct other stationary Gaussian fractional Ornstein-Uhlenbeck fields of the first and second kind.

\section{Proofs and auxiliaries}
\label{sec:proofs}

\begin{proof}[Proof of Theorem \ref{theorem:lamperti}]
Assume that $X$ is stationary. Set $Y_{e^t} = (\mathcal{L}_\Theta X)_{e^t}.$ Then, for $s\in\zn$,
$$Y_{e^{t+s}} = e^{(t+s)\ast \Theta}X_{t+s} \overset{\text{law}}{=}e^{t\ast\Theta+s\ast\Theta} X_t = e^{s\ast\Theta}e^{t\ast\Theta} X_t = e^{s\ast\Theta}Y_{e^t},$$
since the matrices $t\ast\Theta$ and $s\ast\Theta$ commute. Next, assume that $Y$ is $\Theta$-self-similar. Set $X_t = (\mathcal{L}^{-1}_\Theta Y)_t$. Then,  for $s\in\zn$,
$$ X_{t+s} =  e^{- (t+s)\ast\Theta} Y_{e^{t+s}} \overset{\text{law}}{=}  e^{- t\ast\Theta -s\ast\Theta} e^{s\ast\Theta} Y_{e^t} = e^{- t\ast\Theta} Y_{e^t}=X_t $$
again by commutation of the matrix exponents. The finite dimensional distributions can be treated similarly. The fact that $\mathcal{L}_\Theta$ is a bijection follows directly from the definition.
\end{proof}

Next, we state an elementary result for sums of binomial coefficients that we utilize in several occasions. A proof can be found in \cite{voutilainen2023lamperti}.
\begin{lemma}
\label{binomial}
Let $M\in\mathbb{N}$. Then
\begin{equation*}
\sum_{m=0}^M (-1)^m\binom{M}{m} = 0.
\end{equation*}
\end{lemma}

\begin{proof}[Proof of Lemma \ref{lemma:incrementsum}]
If $s_l=t_l$ for some $l$, then the left hand side involves an empty sum and thus, we get $0 = \Delta_{s,t}Z$ agreeing with Remark \ref{rem:properties}. Assume now that $s < t$ element-wise. Also, for the sake of simplicity, we assume that the parameter space of $Z$ is $\zn$. The same proof holds also e.g. for self-similar fields with the parameter space $\{e^t : t\in\zn\}$. Now
\begin{equation}
\label{sumofincrements}
\sum_{j=s+{\bf 1}}^t \Delta_j Z = \sum_{j_1=s_1+1}^{t_1} \dots \sum_{j_N=s_N+1}^{t_N} \sum_{i\in\{0,1\}^N} (-1)^{\sum i} Z_{j_1-i_1,\dots, j_N-i_N}
\end{equation}
consists of terms $Z_{k_1,\dots,k_N}$ with $s_l\leq k_l \leq t_l$ for all $l$. Let $k_l=t_l$ for all $l\in M_+ \subseteq \{1,\dots,N\}$ and $k_l = s_l$ for all $l\in M_- \subseteq \{1,\dots,N\}$. We denote the cardinality of such sets as $|M_-|$. Now $Z_{k_1,\dots,k_N}$ belongs to \eqref{sumofincrements} if and only if
\begin{enumerate}[label=(\roman*)]
\item For all $l\in M_+$: $j_l=t_l$ and $i_l=0$.
\item For all $l\in M_-$: $j_l =s_l$ and $i_l=1$.
\item For all $l\notin M_+ \cup M_-$ either
\begin{enumerate}
\item $j_l = k_l$ and $i_l=0$ 
\item $j_l = k_l+1$ and $i_l=1$
\end{enumerate}
\end{enumerate}
Thus, the total number of terms $Z_{k_1,\dots,k_N}$ in \eqref{sumofincrements} is 
$$ 
(-1)^{|M_-|} \left[ \binom{N-|M_+ \cup M_- |}{0}-  \binom{N-|M_+ \cup M_- |}{1} \pm  \binom{N-|M_+ \cup M_- |}{N-|M_+ \cup M_- |} \right] = 0
$$
whenever $|M_+\cup M_-| \neq N$ by Lemma \ref{binomial}. The sign of the last binomial coefficient depends on the parity of $N-|M_+\cup M_-|$. In the case that $M_+\cup M_- = \{1,\dots,N\}$, the number of terms is $(-1)^{|M_-|}$, where $|M_-|$ is the number of indices for which $k_l = s_l$, and $k_l=t_l$ for all the other indices. Hence, \eqref{sumofincrements} reduces to
\begin{equation*}
\begin{split}
&\sum_{j_1=s_1+1}^{t_1} \dots \sum_{j_N=s_N+1}^{t_N} \sum_{i\in\{0,1\}^N} (-1)^{\sum i} Z_{j_1-i_1,\dots, j_N-i_N}\\
 &= \sum_{(i_1,\dots,i_N)\in\{0,1\}^N} (-1)^{\sum_{l=1}^N i_l} Z_{t_1-i_1(t_1-s_1),\dots,t_N-i_N(t_N-s_N)} = \Delta_{s,t} Z
\end{split}
\end{equation*}
\end{proof}

\begin{proof}[Proof of Lemma \ref{lemma:logarithm}]
The latter claim follows directly from the fact that $\Delta_{{\bf 1}} G = G_{{\bf 1}}$, if $G_t$ vanishes on the hyperplanes where $t_l=0$ for some $l$.

Let $t\in\zn$. We begin the proof of the first claim by showing that $\sup_{j\leq t} \Vert e^{j \ast \Theta} \Delta_j G \Vert < \infty$ almost surely, where $j \leq t$ is understood component-wise. For this, let $k\in \mathbb{N} \cup \{0\}$. We use the abbreviated notation $\sum j = \sum_{l=1}^N j_l$ and $\sum t  = \sum_{l=1}^N t_l $, where in addition $j_l \leq t_l$ for all $l$. Now
\begin{equation}
\label{lars}
\begin{split}
\p \left( \sup_{\substack{\sum j = \sum t -k\\  j\leq t }} \Vert  e^{j \ast \Theta}\Delta_j G \Vert > \epsilon \right) &= \p \left( \bigcup_{\substack{\sum j = \sum t -k\\  j\leq t }} \left \{ \Vert  e^{j \ast \Theta} \Delta_j G \Vert > \epsilon \right\} \right)\\
&\leq \sum_{\substack{\sum j = \sum t -k\\  j\leq t }} \p \left(   \Vert  e^{j \ast \Theta} \Delta_j G \Vert  > \epsilon \right) \\
&\leq  \sum_{\substack{\sum j = \sum t -k\\  j\leq t }} \p \left( \Vert  e^{j \ast \Theta} \Vert \Vert \Delta_j G \Vert > \epsilon \right) \\
&=  \sum_{\substack{\sum j = \sum t -k\\  j\leq t }} \p \left(  \Vert \Delta_{{\bf 1}} G \Vert > \frac{\epsilon}{\Vert  e^{j \ast \Theta} \Vert} \right)
\end{split}
\end{equation}
by stationarity of increments. We set $u_j = \{l:j_l < 0\}$. Note that when $k$ is large enough, at least one of the indices $j_l$ has to be negative. Now
\begin{equation}
\label{tuomas}
\begin{split}
\Vert e^{j\ast \Theta} \Vert &= \Vert \prod_{l=1}^N e^{j_l \Theta_l} \Vert \leq \prod_{l=1}^N \Vert e^{j_l \Theta_l}\Vert = \prod_{l\in u_j} \Vert e^{j_l \Theta_l}\Vert \prod_{l\notin u_j} \Vert e^{j_l \Theta_l}\Vert \leq C_u \prod_{l\in u_j} \Vert e^{j_l \Theta_l}\Vert\\
&\leq C_1\prod_{l\in u_j} \Vert e^{j_l \Theta_l}\Vert,
\end{split}
\end{equation}
since when $l\notin u_j$, it holds that $0 \leq j_l \leq t_l$. Moreover, also the number of possible index sets $u_j$ is finite, so we may simply define $C_1$ in terms of maximums above. Denote the smallest eigenvalue of a positive definite matrix $A$ by $\lambda_1(A)$. In addition, let $\lambda_1 = \min_l \{ \lambda_1(\Theta_l)\}$. Since $j_l < 0$ for $l\in u_j$, this gives
\begin{equation}
\label{iivo}
\prod_{l\in u_j} \Vert e^{j_l \Theta_l}\Vert = \prod_{l\in u_j} e^{j_l \lambda_1(\Theta_l)} = e^{\sum_{l\in u_j} j_l \lambda_1(\Theta_l)} \leq e^{\lambda_1 \sum_{l\in u_j} j_l } \leq e^{\lambda_1 \sum_{l=1}^N j_l } = e^{\lambda_1\left( \sum_{l=1}^N t_l -k \right) }.
\end{equation}
Combination of \eqref{lars}, \eqref{tuomas} and \eqref{iivo} gives
\begin{equation}
\label{mika}
\p \left( \sup_{\substack{\sum j = \sum t -k\\  j\leq t }} \Vert  e^{j \ast \Theta} \Delta_j G \Vert > \epsilon \right) \leq \sum_{\substack{\sum j = \sum t -k\\  j\leq t }} \p \left( \Vert \Delta_{{\bf 1}}G \Vert > \frac{\epsilon}{C_1 e^{\lambda_1 (\sum t -k)}} \right),
\end{equation}
where the summands are now independent of the summation index and consequently, we next evaluate the number of multi-indices $j$ with $j \leq t$ and $\sum j = \sum t -k$. Set $ \tilde{j} = j-t$. Then equivalently, we want to count $\tilde{j}$ such that $\tilde{j} \leq 0$ and $\sum \tilde{j} = \sum (j-t) = -k$. This corresponds to finding the number of \emph{weak compositions of $k$ into $N$ parts}, and that number is
$$\binom{k+N-1}{k} = \frac{(k+N-1)!}{(N-1)!k!} = \mathcal{O}(k^{N-1}).$$
When $k$ is large enough, this together with \eqref{mika} yields
\begin{equation*}
\begin{split}
&\p \left( \sup_{\substack{\sum j = \sum t -k\\  j\leq t }} \Vert  e^{j \ast \Theta} \Delta_j G \Vert > \epsilon \right) \leq C_2 k^{N-1}\p \left( \Vert \Delta_{{\bf 1}}G \Vert > \frac{\epsilon}{C_1 e^{\lambda_1 (\sum t -k)}} \right)\\
&= C_2 k^{N-1}\p \left( \Vert \Delta_{{\bf 1}}G \Vert \mathbbm{1}_{\{\Vert  \Delta_{{\bf 1}}G \Vert \neq 0 \} } > \frac{\epsilon}{C_1 e^{\lambda_1 (\sum t -k)}} \right)\\
&= C_2 k^{N-1}\p \left( \ln \Vert \Delta_{{\bf 1}}G \Vert \mathbbm{1}_{\{\Vert  \Delta_{{\bf 1}}G \Vert \neq 0 \} } > \ln \epsilon - \ln C_1 - \lambda_1\left(\sum t - k\right) \right),
\end{split}
\end{equation*}
where, when $k$ is large enough, 
$$\ln \epsilon - \ln C_1 - \lambda_1\sum t +\lambda_1 k = k\left(\frac{\ln \epsilon - \ln C_1 - \lambda_1\sum t}{k} + \lambda_1 \right) \geq C_3k $$
for some $C_3 > 0$. Hence, for $k$ large enough,
\begin{equation*}
\begin{split}
&\p \left( \sup_{\substack{\sum j = \sum t -k\\  j\leq t }} \Vert  e^{j \ast \Theta} \Delta_j G \Vert > \epsilon \right) \leq C_2 k^{N-1}\p \left( \ln \Vert \Delta_{{\bf 1}}G \Vert \mathbbm{1}_{\{\Vert  \Delta_{{\bf 1}}G \Vert \neq 0 \} } > C_3 k \right)\\
&= C_2 k^{N-1}\p \left( \ln \Vert \Delta_{{\bf 1}}G \Vert \mathbbm{1}_{\{\Vert  \Delta_{{\bf 1}}G \Vert > 1 \} } > C_3 k \right)\\
&\leq C_2 k^{N-1} \frac{\e \left (\ln \Vert \Delta_{{\bf 1}}G \Vert \mathbbm{1}_{\{\Vert  \Delta_{{\bf 1}}G \Vert > 1 \} } \right)^{N+\delta}}{(C_3k)^{N+\delta}} = \frac{C}{k^{1+\delta}}
\end{split}
\end{equation*}
by Markov's inequality and \eqref{logarithmic}. Then, Borel-Cantelli lemma gives
$$\sup_{\substack{\sum j = \sum t -k\\  j\leq t }} \Vert  e^{j \ast \Theta} \Delta_j G \Vert  \to 0$$ almost surely as $k\to\infty$,
and furthermore
\begin{equation*}
\label{bounded}
\sup_{k \geq 0} \sup_{\substack{\sum j = \sum t -k\\  j\leq t }} \Vert  e^{j \ast \Theta} \Delta_j G \Vert = \sup_{j\leq t} \Vert  e^{j \ast \Theta} \Delta_j G \Vert < \infty
\end{equation*}
almost surely, where $\Theta\in\poc$ is arbitrary.

Now we are in a position to show that \eqref{eq:glimitcondition} converges almost surely. We have that
\begin{equation}
\label{kerttu}
\begin{split}
\sum_{j=-M}^t \Vert e^{j \ast \Theta} \Delta_j G \Vert &\leq \sum_{j=-M}^t \Vert e^{\frac{1}{2}j \ast \Theta} \Delta_j G \Vert \Vert e^{\frac{1}{2}j \ast \Theta} \Vert \leq \sup_{j\leq t} \Vert  e^{\frac{1}{2}j \ast \Theta} \Delta_j G \Vert  \sum_{j=-M}^t \Vert e^{\frac{1}{2}j \ast \Theta} \Vert \\
&\leq C \sum_{j=-M}^t \Vert e^{\frac{1}{2}j \ast \Theta} \Vert
\end{split}
\end{equation}
almost surely. We write the above sum in $2^N$ parts as 
\begin{equation}
\label{saku}
\sum_{j=-M}^t \Vert e^{\frac{1}{2}j \ast \Theta} \Vert = \sum_{u\subseteq \{1,\dots,N\}} \sum_{j_u = -M_u}^{{\bf -1}} \sum_{j_{-u} = {{\bf 0}}}^{t_u} \Vert e^{\frac{1}{2}j \ast \Theta} \Vert,
\end{equation}
where $j_l$ runs from $-M_l$ to $-1$ for  $l\in u$ and from $0$ to $t_l$ for $l\notin u$. Now
\begin{equation*}
\begin{split}
\Vert e^{\frac{1}{2}j \ast \Theta} \Vert &= \Vert e^{\sum_{l\in u}\frac{1}{2} j_l \Theta_l+ \sum_{l\notin u}\frac{1}{2} j_l \Theta_l} \Vert \leq \Vert e^{\sum_{l\in u}\frac{1}{2} j_l \Theta_l} \Vert \Vert e^{\sum_{l\notin u}\frac{1}{2} j_l \Theta_l} \Vert \leq C \Vert e^{\sum_{l\in u}\frac{1}{2} j_l \Theta_l} \Vert\\
&\leq C \prod_{l\in u} \Vert e^{\frac{1}{2} j_l \Theta_l} \Vert = C\prod_{l\in u} e^{\frac{1}{2} j_l \lambda_1(\Theta_l)}
\end{split}
\end{equation*}
similarly as in the earlier part of the proof. By combining with \eqref{kerttu} and \eqref{saku},
\begin{equation*}
\sum_{j=-M}^t \Vert e^{j \ast \Theta} \Delta_j G \Vert \leq C \sum_{u\subseteq \{1,\dots,N\}} \sum_{j_{-u} = {{\bf 0}}}^{t_u}\sum_{j_u = -M_u}^{{\bf -1}} \prod_{l\in u}   e^{\frac{1}{2} j_l \lambda_1(\Theta_l)} 
\end{equation*}
almost surely. As $M \to \infty$,
$$\lim_{M\to\infty} \sum_{j=-M}^t \Vert e^{j \ast \Theta} \Delta_j G \Vert  \leq C\sum_{u\subseteq \{1,\dots,N\}}\sum_{j_{-u} = {{\bf 0}}}^{t_u} C_u,$$
where the remaining sums are over finite sets showing that the series is absolutely convergent almost surely and thus, completing the proof.
\end{proof}

\begin{proof}[Proof of Corollary \ref{kor:moment}]
Let $M>1$ be such that $(\ln x)^{N+\delta} < x$, when $x\geq M$. Then
\begin{equation*}
\begin{split}
&\e \left(\ln \Vert \Delta_{{\bf 1}} G \Vert \mathbbm{1}_{\{ \Vert \Delta_{{\bf 1}} G\Vert > 1 \}} \right)^{N+ \delta}\\ &= \e \left(\ln \Vert \Delta_{{\bf 1}} G \Vert \mathbbm{1}_{\{ \Vert \Delta_{{\bf 1}} G\Vert \geq M \}} \right)^{N+ \delta}+ \e \left(\ln \Vert \Delta_{{\bf 1}} G \Vert \mathbbm{1}_{\{ 1 <\Vert \Delta_{{\bf 1}} G\Vert < M \}} \right)^{N+ \delta}\\
&\leq \e \Vert \Delta_{{\bf 1}} G\Vert + (\ln M)^{N+\delta}
\end{split}
\end{equation*}
\end{proof}

\begin{proof}[Proof of Lemma \ref{lemma:selfsimilar}]
The unit cube increments of $Y$ are of the form
\begin{equation}
\label{nfold}
\Delta_t Y = \sum_{i\in\{0,1\}^N} (-1)^{\sum i} \sum_{j=-\infty}^{t-i} e^{j \ast \Theta} \Delta_jG =  \sum_{i\in\{0,1\}^N} (-1)^{\sum i} \sum_{j_1=-\infty}^{t_1-i_1} \dots \sum_{j_N=-\infty}^{t_N-i_N}e^{j \ast \Theta} \Delta_j G.
\end{equation}
Let us consider $e^{j\ast \Theta} \Delta_j G$, where $j_l = t_l$ for all $l\in M_+ \subset \{1,\dots, N\}$ with $|M_+| = m \neq N$. The term belongs to the $N$-fold sum of  \eqref{nfold} for every $i$ such that $i_l=0$ for all $l\in M_+$. That is, $m$ of the indices $i_l$ are zeros and remaining $N-m$ indices can be zeros or ones. Taking into account of the alternating sign in \eqref{nfold}, the total number of terms $e^{j\ast \Theta} \Delta_j G$ in \eqref{nfold} is
$$\binom{N-m}{0} - \binom{N-m}{1}+ \dots + \binom{N-m}{N-m} =0, \quad N-m\text{ is even}$$
and 
$$\binom{N-m}{0} - \binom{N-m}{1}+ \dots - \binom{N-m}{N-m} =0, \quad N-m\text{ is odd}$$
by Lemma \ref{binomial}. If $m=N$, then the corresponding term $e^{t\ast \Theta} \Delta_t G$ belongs only to the $N$-fold sum with $i= {\bf 0}$ completing the first part of the proof.\\
For the second part, we notice that
$$ \sum_{j_1=-M_1}^{t_1}\dots  \sum_{j_N=-M_N}^{t_N}  e^{j\ast \Theta} \Delta_j G =  \sum_{j_1=0}^{t_1+M_1}\dots  \sum_{j_N=0}^{t_N+M_N}  e^{(t-j)\ast \Theta} \Delta_{t-j} G = e^{t\ast \Theta} \sum_{j_1=0}^{t_1+M_1}\dots  \sum_{j_N=0}^{t_N+M_N}  e^{-j\ast \Theta} \Delta_{t-j} G.$$
Furthermore
$$e^{(t+s)\ast \Theta} \sum_{j_1=0}^{t_1+s_1+M_1}\dots  \sum_{j_N=0}^{t_N+s_N+M_N}  e^{-j\ast \Theta} \Delta_{t+s-j} G \yli e^{s\ast \Theta}e^{t\ast \Theta}\sum_{j_1=0}^{t_1+s_1+M_1}\dots  \sum_{j_N=0}^{t_N+s_N+M_N}  e^{-j\ast \Theta} \Delta_{t-j} G.$$
Since $G\in\mathcal{G}_\Theta$, this yields
\begin{equation*}
\begin{split}
Y_{e^{t+s}} &= \lim_{M_1\to\infty} \dots \lim_{M_N\to\infty}e^{(t+s)\ast \Theta} \sum_{j_1=0}^{t_1+s_1+M_1}\dots  \sum_{j_N=0}^{t_N+s_N+M_N}  e^{-j\ast \Theta} \Delta_{t+s-j} G\\
 &\yli \lim_{M_1\to\infty} \dots \lim_{M_N\to\infty} e^{s\ast \Theta}e^{t\ast \Theta}\sum_{j_1=0}^{t_1+s_1+M_1}\dots  \sum_{j_N=0}^{t_N+s_N+M_N}  e^{-j\ast \Theta} \Delta_{t-j} G\\
 &= e^{s\ast \Theta}  \lim_{\tilde{M}_1\to\infty} \dots \lim_{\tilde{M}_N\to\infty} e^{t\ast \Theta}\sum_{j_1=0}^{t_1+\tilde{M}_1}\dots  \sum_{j_N=0}^{t_N+\tilde{M}_N}  e^{-j\ast \Theta} \Delta_{t-j} G = e^{s\ast\Theta} Y_{e^t},
\end{split}
\end{equation*}
where $\tilde{M} = M+s$. Treating multidimensional distributions similarly completes the proof.
\end{proof}

\begin{proof}[Proof of Lemma \ref{lemma:stationaryincrements}]
First we show that the increments of $G$ are stationary under the assumption $\Theta \in \poc$. Let $s\in\zn$. Then
$$ \Delta_{t+s}G = e^{-(t+s)\ast \Theta} \Delta_{t+s}Y \overset{\text{law}}{=} e^{-(t+s)\ast \Theta}e^{s\ast \Theta} \Delta_t Y = e^{-t\ast \Theta}\Delta_t Y = \Delta_t G,$$
where we have used $\Theta$-self-similarity of $Y$ and the fact that the involved matrix exponents commute.
Multidimensional distributions of $\Delta G$ can be treated similarly.\\
Next, we consider the convergence of the $N$-fold sum. We have that
\begin{equation}
\label{Gsum}
\sum_{j_1 = -M_1+1}^{t_1} \dots \sum_{j_N = -M_N+1}^{t_N} e^{j\ast\Theta} \Delta_j G = \sum_{j_1 = -M_1+1} ^{t_1}\dots \sum_{j_N = -M_N+1}^{t_N} \Delta_t Y = \Delta_{-M,t} Y
\end{equation}
by Lemma \ref{lemma:incrementsum}. Furthermore
\begin{equation}
\label{incrementY}
\Delta_{-M,t} Y = \sum_{i\in\{0,1\}^N} (-1)^{\sum i} Y_{e^{t-i(t+M)}},
\end{equation}
where
\begin{equation}
\label{law}
 Y_{e^{t-i(t+M)}} \yli  e^{-i(t+M)\ast \Theta}Y_{e^{t}}
\end{equation}
and
\begin{equation}
\label{norms}
\Vert e^{-i(t+M) \ast \Theta} Y_{e^t}\Vert \leq \Vert e^{-i(t+M) \ast \Theta}\Vert \Vert  Y_{e^t}\Vert = \Vert e^{-\sum_{l=1}^N i_l(t_l+M_l) \Theta_l}\Vert \Vert  Y_{e^t}\Vert. 
\end{equation}
The matrices $i_l(t_l+M_l)\Theta_l$ are positive definite when $i_l=1$ and $M_l$ is large enough. Let $\lambda_1(A)$ denote the smallest eigenvalue of matrix $A$. As a direct consequence of Courant-Fischer min-max theorem
$$\lambda_1(A+B) \geq \max\{ \lambda_1(A), \lambda_1(B)\}$$
for any positive semidefinite matrices $A$ and $B$. Hence, we get
$$\lambda_1\left(\sum_{l=1}^N i_l(t_l+M_l) \Theta_l \right) \geq \max_l \{ \lambda_1(i_l(t_l+M_l)\Theta_l)\}.$$
Let $\sum_{l=1}^N i_l(t_l+M_l) \Theta_l = Q\Lambda Q^T$ be an eigendecomposition with eigenvalues $\lambda_1 \leq \dots \leq \lambda_n$. Then 
\begin{equation}
\label{tozero}
\begin{split}
\Vert e^{-\sum_{l=1}^N i_l(t_l+M_l) \Theta_l} \Vert  &= \Vert Q e^{-\Lambda} Q^T\Vert = \Vert \mathrm{diag}(e^{-\lambda_i}) \Vert= \max_i \{e^{-\lambda_i}\} = e^{-\lambda_1} \\ 
&\leq e^{-\max_l \{ \lambda_1(i_l(t_l+M_l)\Theta_l)\}} = e^{-\max_l \{ i_l(t_l+M_l)\lambda_1(\Theta_l)\}} \to 0,
\end{split}
\end{equation}
whenever $i\neq {\bf 0}$ and $M\to \infty$. Now, from \eqref{law}, \eqref{norms} and \eqref{tozero}, 
\begin{equation*}
\p \left(\Vert  Y_{e^{t-i(t+M)}} \Vert \geq \epsilon \right) = \p \left( \Vert e^{-i(t+M)\ast \Theta}Y_{e^{t}} \Vert \geq \epsilon \right) \to 0,
\end{equation*}
whenever $i \neq {\bf 0}$ and $M\to\infty$.
By combining this with \eqref{Gsum} and \eqref{incrementY}, we conclude that
\begin{equation*}
\begin{split}
\sum_{j=-M+{\bf 1}}^t e^{j\ast \Theta} \Delta_j G = \sum_{i\in\{0,1\}^N} (-1)^{\sum i} Y_{e^{t-i(t+M)}} \to Y_{e^t}
\end{split}
\end{equation*}
in probability as $M\to\infty$.
\end{proof}

\begin{proof}[Proof of Lemma \ref{lemma:increments}]
The property that $G_t = 0$ on the hyperplanes with $t_l=0$ for some $l$ follows from the lower bound $j_u=1$ of the summation in \eqref{G} resulting in an empty sum. Let $t\in\mathbb{Z}^N$ and consider a term $e^{-j\ast\Theta} \Delta_j Y$. We notice that it belongs to \eqref{G} if and only if $1 \leq j_l \leq t_l$ for all $l\in u$ and $t_l+1 \leq j_l \leq 0$ for all $l\in -u$. In what follows, we write $\{t-i \geq 0\} = \{l : t_l-i_l \geq 0 \}$. Now 
\begin{equation}
\label{deltaG}
\Delta_t G = \sum_{i\in \{0,1\}^N} (-1)^{\sum i} (-1)^{|\{t-i < 0\}|} \sum_{j_{\{t-i \geq 0\}} ={\bf 1}}^{(t-i)_{\{t-i \geq 0\}}} \sum_{{j_{\{t-i < 0\}} =(t-i)_{\{t-i < 0\}}+{\bf 1}}}^{\bf 0} e^{-j\ast\Theta} \Delta_j Y.
\end{equation}
By the previous observation, \eqref{deltaG} consists of terms $e^{-j\ast\Theta} \Delta_j Y$ for which $1\leq j_l \leq t_l$ for all $l$ such that $t_l \geq 1$, $j_l =0$ for all $l$ such that $t_l =0$ and $t_l \leq j_l \leq 0$ for all $l$ such that $t_l \leq -1$. Let $e^{-j\ast\Theta} \Delta_j Y$ be such term. Moreover, let $M_+, M_0$ and $M_-$ be the set of indices $l$ for which $t_l\geq 1, t_l=0$ and $t_l \leq -1$, respectively. If $e^{-j\ast\Theta} \Delta_j Y$ belongs to the $N$-fold sum in \eqref{deltaG} for $i\in\{0,1\}^N$, then the following observations hold true.
\begin{enumerate}[label=(\roman*)]
\item It holds that $i_{M_0} = 1$. That is, $|M_0|$ of the indices $i_l$ are ones. The corresponding contribution to the sign in \eqref{deltaG} is $(-1)^{|M_0|}(-1)^{|M_0|} =1$.
\item Let $m_+$ be the set of indices for which $j_l = t_l$ with $t_l\geq 1$. Now, it holds that $i_{m_+} = 0$. That is, $|m_+|$ of the indices are zeros having no contribution to the sign in \eqref{deltaG}. The indices $M_+-m_+$ can be either zeros or ones.
\item  Let $m_-$ be the set of indices for which $j_l = t_l$ with $t_l\leq -1$. Now, it holds that $i_{m_-} = 1$. That is, $|m_-|$ of the indices are ones. The contribution to the sign in \eqref{deltaG} is $(-1)^{|m_-|}(-1)^{|M_-|} = (-1)^{|m_-|+|M_-|}$. The indices $M_--m_-$ can be either zeros or ones.
\end{enumerate}
Now the total number of terms  $e^{-j\ast\Theta} \Delta_j Y$ in \eqref{deltaG} is 
\begin{equation}
\label{number}
\begin{split}
&(-1)^{|m_-|+|M_-|}\left[\binom{|M_+-m_+|+|M_--m_-|}{0}\right.\\
&\left.-\binom{|M_+-m_+|+|M_--m_-|}{1}
+\dots \pm \binom{|M_+-m_+|+|M_--m_-|}{|M_+-m_+|+|M_--m_-|}\right],
\end{split}
\end{equation}
where the sign of the last binomial coefficient depends on whether $|M_+-m_+|+|M_--m_-|$ is even or odd. Regardless, by Lemma \ref{binomial}, the result is zero unless $|M_+-m_+|+|M_--m_-|=0$. In that case $M_+ = m_+$ and $M_-= m_-$ meaning that $j_l=t_l$ for all $l$, i.e. $e^{-j\ast\Theta} \Delta_j Y=e^{-t\ast\Theta} \Delta_t Y$. Now \eqref{number} gives
$$(-1)^{|m_-|+|M_-|}\binom{0}{0}= (-1)^{2|M_-|} = 1.$$
\end{proof}

\begin{proof}[Proof of Theorem \ref{theorem:bijection}]
The first claim of the theorem is obtained by combining Lemmas \ref{lemma:selfsimilar}, \ref{lemma:stationaryincrements} and \ref{lemma:increments}. Let us consider $(\mathcal{M}_\Theta^{-1}\circ\mathcal{M}_\Theta) (Y)$ for a $\Theta$-self-similar $Y$. Then, by Lemmas \ref{lemma:incrementsum} and \ref{lemma:increments}
\begin{equation*}
\label{direction1}
\begin{split}
\mathcal{M}_\Theta^{-1}(\mathcal{M}_\Theta (Y))_t &= \lim_{M_1\to\infty} \dots \lim_{M_N\to\infty} \sum_{j=-M+{\bf 1}}^t e^{j\ast\Theta} \Delta_j (\mathcal{M}_\Theta(Y))\\
 &=\lim_{M_1\to\infty} \dots \lim_{M_N\to\infty} \sum_{j=-M+{\bf 1}}^t  \Delta_j Y = \lim_{M_1\to\infty} \dots \lim_{M_N\to\infty} \Delta_{-M,t} Y = Y_{e^t}\\
\end{split}
\end{equation*}
similarly as in the proof of Lemma \ref{lemma:stationaryincrements}.

Next, we consider $(\mathcal{M}_\Theta \circ \mathcal{M}_\Theta^{-1}) (G)$ for $G\in\mathcal{G}_{\Theta,0}$. By Lemma \ref{lemma:selfsimilar}, $Y_{e^t} = (\mathcal{M}_\Theta^{-1} G)_{e^t}$ is $\Theta$-self-similar with $\Delta_t Y = e^{t\ast \Theta} \Delta_t G$. On the other hand, $\tilde{G}_t = (\mathcal{M}_\Theta Y)_t$ is in $\mathcal{G}_{\Theta,0}$ with $\Delta_t \tilde{G} = e^{-t\ast\Theta} \Delta_t Y$. Since $e^{-t\ast\Theta}$ is the inverse of $e^{t\ast\Theta}$, we get
$$e^{t\ast\Theta} \Delta_t \tilde{G} = \Delta_t Y = e^{t\ast\Theta} \Delta_t G$$
and furthermore $ \Delta_t \tilde{G} = \Delta_t G$ for all $t\in\zn$. It remains to show that this implies $\tilde{G}_t = G_t$ for all $t\in\zn$. For this, let $t$ be such that $t_l\neq 0$ for all $l$. Then, by Remark \ref{rem:properties},
$$G_t = \Delta_{{\bf 0},t} G = (-1)^m \Delta_{\tilde{{\bf 0}}, \tilde{t}} G,$$
where $m$ is the number of elements for which $t_l <0$ and $\tilde{{\bf 0}}$ and $\tilde{t}$ are the vectors where such elements are swapped  so that $\tilde{t} \geq \tilde{{\bf 0}}$ element-wise. Therefore, by Lemma \ref{lemma:incrementsum},
$$G_t =  (-1)^m \Delta_{\tilde{{\bf 0}}, \tilde{t}} G =  (-1)^m\sum_{j= \tilde{{\bf 0}} +{\bf 1}}^{\tilde{t}} \Delta_j G =  (-1)^m\sum_{j= \tilde{{\bf 0}} +{\bf 1}}^{\tilde{t}} \Delta_j \tilde{G} = \tilde{G}_t.$$
\end{proof}

\begin{proof}[Proof of Theorem \ref{theorem:first}]
Set $Q_t^{(N)} = \Delta_t G$. Then, by Definition \ref{defi:hassu}
\begin{equation*}
\begin{split}
X_t &= \hat{\Theta}\star\hat{X}_t + \Delta_t G =  \sum_{\substack{i\in\{0,1\}^N\\ i\neq \bf{0}}} (-1)^{1+\sum i} e^{- i\ast\Theta} X_{t-i}+Q_t^{(N)}\\
&=  \sum_{\substack{i\in\{0,1\}^{N-1}}} (-1)^{\sum i} e^{- (i,1)\ast\Theta} X_{t-i,t_{N}-1} +\sum_{\substack{i\in\{0,1\}^{N-1}\\ i\neq \bf{0}}} (-1)^{1+\sum i} e^{- (i,0)\ast\Theta} X_{t-i,t_N}+Q_t^{(N)},
\end{split}
\end{equation*}
where e.g. $(i,1) = (i_1,\dots,i_{N-1},1)$ and $X_{t-i,t_N-1} = X_{t_1-i_1,\dots, t_{N-1}-i_{N-1}, t_N-1}$. From above we get
\begin{equation*}
\begin{split}
X_t + \sum_{\substack{i\in\{0,1\}^{N-1}\\ i\neq \bf{0}}} (-1)^{\sum i} e^{- (i,0)\ast\Theta} X_{t-i,t_N} &= \sum_{\substack{i\in\{0,1\}^{N-1}}} (-1)^{\sum i} e^{- (i,0)\ast\Theta} X_{t-i,t_N}\\
&=  \sum_{\substack{i\in\{0,1\}^{N-1}}} (-1)^{\sum i} e^{- (i,1)\ast\Theta} X_{t-i,t_{N}-1} +Q_t^{(N)}\\
&= e^{-\Theta_N} \sum_{\substack{i\in\{0,1\}^{N-1}}} (-1)^{\sum i} e^{- (i,0)\ast\Theta} X_{t-i,t_{N}-1} +Q_t^{(N)},
\end{split}
\end{equation*}
since the involved matrices commute. By denoting
$$Y_t^{(N)} (t_N) =  \sum_{\substack{i\in\{0,1\}^{N-1}}} (-1)^{\sum i} e^{- (i,0)\ast\Theta} X_{t-i,t_N} $$
we obtain
$$Y_t^{(N)} (t_N) = e^{-\Theta_N} Y_t^{(N)} (t_N-1) + Q_{t,t_N}^{(N)}. $$
By iterating the above equation, we get for every $n\in\mathbb{N}$ that
\begin{equation}
\begin{split}
\label{sumstays}
Y_t^{(N)} (t_N) &= e^{-(n+1)\Theta_N} Y_t^{(N)} (t_N-n-1) + \sum_{j_N = 0}^n e^{-j_N \Theta_N} Q_{t,t_N-j_N}^{(N)}\\
&=e^{-(n+1)\Theta_N} Y_t^{(N)} (t_N-n-1) + \sum_{k_N = t_N-n}^{t_N} e^{(k_N-t_N) \Theta_N} Q_{t,k_N}^{(N)},
\end{split}
\end{equation}
where
$$e^{-(n+1)\Theta_N} Y_t^{(N)} (t_N-n-1) = e^{-t_N \Theta_N} e^{(t_N-n-1)\Theta_N}  Y_t^{(N)} (t_N-n-1) =  e^{-t_N\Theta_N}e^{m\Theta_N}Y_t^{(N)} (m) $$
with $m = t_N-n-1\to -\infty$ as $n\to\infty$. By the condition (i) of Theorem \ref{theorem:first},

$$e^{m\Theta_N}Y_t^{(N)} (m) = e^{m\Theta_N} \sum_{\substack{i\in\{0,1\}^{N-1}}} (-1)^{\sum i} e^{- (i,0)\ast\Theta} X_{t-i,m} \to 0$$
in probability as $m\to -\infty$. Hence, as $n\to\infty$, \eqref{sumstays} yields
$$ Y_t^{(N)} (t_N) = e^{-t_N \Theta_N} \sum_{j_N = -\infty}^{t_N} e^{j_N \Theta_N} Q_{t,j_N}^{(N)}. $$
Thus, the induction assumption that for $k\in\mathbb{N}$,
\begin{equation}
\begin{split}
\label{inductionassumption}
 &\sum_{i\in\{0,1\}^{N-k}} (-1)^{\sum i} e^{- (i,{\bf 0})\ast\Theta} X_{t-i,t_{N-k+1},\dots,t_N}\\
  &= e^{-\sum_{l=1}^k t_{N-l+1}\Theta_{N-l+1}} \sum_{j_{N-k+1} =-\infty}^{t_{N-k+1}} \dots \sum_{j_N =-\infty}^{t_N}  e^{\sum_{l=1}^k j_{N-l+1}\Theta_{N-l+1}}Q_{t,j_{N-k+1},\dots,j_N}^{(N)} \eqqcolon Q_t^{(N-k)}
\end{split}
\end{equation}
holds true when $k=1$. Note that above $(i,{\bf 0}) = (i_1,\dots,i_{N-k},0, \dots, 0)$. This gives
\begin{equation*}
\begin{split}
&\sum_{i\in\{0,1\}^{N-k-1}} (-1)^{\sum i} e^{- (i,{\bf 0})\ast\Theta} X_{t-i,t_{N-k},\dots,t_N}\\
&= e^{-\Theta_{N-k}} \sum_{i\in\{0,1\}^{N-k-1}} (-1)^{\sum i} e^{- (i,{\bf 0})\ast\Theta} X_{t-i,t_{N-k}-1, t_{N-k+1},\dots,t_N}+Q_t^{(N-k)}.
\end{split}
\end{equation*}
Thus, by setting
$$Y_t^{(N-k)} (t_{N-k}) =\sum_{i\in\{0,1\}^{N-k-1}} (-1)^{\sum i} e^{- (i,{\bf 0})\ast\Theta} X_{t-i,t_{N-k},\dots,t_N} $$
we obtain that
$$ Y_t^{(N-k)} (t_{N-k}) = e^{-\Theta_{N-k}}Y_t^{(N-k)} (t_{N-k}-1) +Q_{t_1,\dots,t_{N-k},\dots,t_N}^{(N-k)}.$$
Again by iteration, it holds for every $n\in\mathbb{N}$ that
\begin{equation*}
\begin{split}
Y_t^{(N-k)} (t_{N-k}) &= e^{-(n+1)\Theta_{N-k}}Y_t^{(N-k)} (t_{N-k}-n-1)\\
 &+\sum_{j_{N-k} = 0}^n e^{-j_{N-k} \Theta_{N-k}} Q_{t_1,\dots,t_{N-k-1},t_{N-k}-j_{N-k},t_{N-k+1},\dots,t_N}^{(N-k)},
\end{split}
\end{equation*}
where similarly as before
$$e^{-(n+1)\Theta_{N-k}}Y_t^{(N-k)} (t_{N-k}-n-1) \overset{\p}{\longrightarrow} 0\quad\text{as } n\to\infty. $$
Hence
\begin{equation*}
\begin{split}
Y_t^{(N-k)} (t_{N-k}) &= \sum_{j_{N-k} = 0}^\infty e^{-j_{N-k} \Theta_{N-k}} Q_{t_1,\dots,t_{N-k-1},t_{N-k}-j_{N-k},t_{N-k+1},\dots,t_N}^{(N-k)}\\
&= e^{-t_{N-k} \Theta_{N-k}} \sum_{j_{N-k} = -\infty}^{t_{N-k}}e^{j_{N-k} \Theta_{N-k}} Q_{t_1,\dots,t_{N-k-1},j_{N-k},t_{N-k+1},\dots,t_N}^{(N-k)}.
\end{split}
\end{equation*}
By the definitions of $Y_t^{(N-k)}$ and $Q_t^{(N-k)}$ this gives directly that
\begin{equation*}
\begin{split}
&\sum_{i\in\{0,1\}^{N-k-1}} (-1)^{\sum i} e^{- (i,{\bf 0})\ast\Theta} X_{t-i,t_{N-k},\dots,t_N}\\
&=  e^{-\sum_{l=1}^{k+1} t_{N-l+1}\Theta_{N-l+1}} \sum_{j_{N-k} =-\infty}^{t_{N-k}} \dots \sum_{j_N =-\infty}^{t_N}  e^{\sum_{l=1}^{k+1} j_{N-l+1}\Theta_{N-l+1}}Q_{t,j_{N-k},\dots,j_N}^{(N)},
\end{split}
\end{equation*}
which completes the induction step. Now, by choosing $k=N$ in \eqref{inductionassumption} we get
\begin{equation*}
\begin{split}
X_t &=  e^{-\sum_{l=1}^{N} t_{N-l+1}\Theta_{N-l+1}} \sum_{j =-\infty}^{t}  e^{\sum_{l=1}^{N} j_{N-l+1}\Theta_{N-l+1}}Q_{j}^{(N)}\\ 
&=   e^{- t\ast\Theta} \sum_{j =-\infty}^{t}  e^{j\ast \Theta}\Delta_{j}G = (\mathcal{L}_\Theta^{-1} ( \mathcal{M}_\Theta^{-1}G))_t
\end{split}
\end{equation*}
and thus, $X$ is stationary.
\end{proof}

We utilize the present value and the corresponding increment in order to define the notion of ''previous value'' of a field. 
\begin{defi}
Let $X = (X_t)_{t\in \zn}$ and $t\in\zn$. We call the random vector
$$X_t^-  \coloneqq X_t - \Delta_t X = \sum_{\substack{(i_1,\ldots,i_N)\in\{0,1\}^N\\ (i_1,\ldots,i_N)\neq \bf{0}}} (-1)^{1+\sum_{l=1}^N i_l} X_{t_1-i_1,\dots,t_N-i_N}.$$
as the ''previous value'' of $X$ at $t$.
\end{defi}

\begin{proof}[Proof of Theorem \ref{theorem:second}]
Let $Y = \mathcal{L}_\Theta X$. Then
\begin{equation*}
\begin{split}
\Delta_t X = X_t -X_t^- &= e^{-t \ast \Theta} Y_{e^t} - X_t^- -e^{-t \ast \Theta}\Delta_t Y +e^{-t \ast \Theta}\Delta_t Y \\
&=  e^{-t \ast \Theta}(Y_{e^t} - \Delta_t Y) - X_t^- +e^{-t \ast \Theta}\Delta_t Y
\end{split}
\end{equation*}
giving
\begin{equation}
\label{mainii}
\Delta_t X + X_t^- = X_t =  e^{-t \ast \Theta}(Y_{e^t} - \Delta_t Y)+ e^{-t \ast \Theta}\Delta_t Y.
\end{equation}
Define $G = \mathcal{M}_\Theta(Y) = (\mathcal{M}_\Theta \circ  \mathcal{L}_\Theta) (X)$. Then $\Delta_t G = e^{-t \ast \Theta}\Delta_t Y$. In addition, since $\Theta \in \poc$,
\begin{equation*}
\begin{split}
e^{-t \ast \Theta}(Y_{e^t} - \Delta_t Y) &= e^{-t \ast \Theta} \sum_{\substack{i\in\{0,1\}^N\\ i\neq \bf{0}}} (-1)^{1+\sum i}Y_{e^{t-i}}\\
 &= \sum_{\substack{i\in\{0,1\}^N\\ i\neq \bf{0}}} (-1)^{1+\sum i} e^{-i \ast \Theta}e^{-(t-i) \ast \Theta}Y_{e^{t-i}}\\
 &= \sum_{\substack{i\in\{0,1\}^N\\ i\neq \bf{0}}} (-1)^{1+\sum i} e^{-i \ast \Theta}X_{t-i} =  \hat{\Theta}\star\hat{X}_t.
\end{split}
\end{equation*}
From \eqref{mainii} we conclude that
$$X_t = \hat{\Theta}\star\hat{X}_t + \Delta_t G$$
with $G= (\mathcal{M}_\Theta \circ  \mathcal{L}_\Theta) (X)$.
\end{proof}

\begin{proof}[Proof of Theorem \ref{theorem:main}]
By Theorems \ref{theorem:first} and \ref{theorem:second}, it remains to prove uniqueness of $G$ in \ref{theorem:second}. Let $X$ be stationary and $G, \tilde{G}\in\mathcal{G}_{\Theta,0}$ satisfy \eqref{ar1}. Then, by Theorem \ref{theorem:first},
$$ (\mathcal{L}_\Theta^{-1} \circ \mathcal{M}_\Theta^{-1})(G) = X = (\mathcal{L}_\Theta^{-1} \circ \mathcal{M}_\Theta^{-1})(\tilde{G})$$
giving $$ G = (\mathcal{M}_\Theta \circ \mathcal{L}_\Theta)(X) = \tilde{G}.$$
\end{proof}

\bibliographystyle{plain}
\bibliography{biblio}


\end{document}